\documentclass{amsart}
\usepackage{amssymb}
\usepackage{mathrsfs}
\usepackage{hyperref}
\usepackage[curve,arrow,matrix]{xy}

\theoremstyle{plain}
\newtheorem{thm}{Theorem}[section]
\newtheorem{prop}[thm]{Proposition}
\newtheorem{lemma}[thm]{Lemma}

\newtheorem{cor}[thm]{Corollary}

\theoremstyle{definition}
\newtheorem{defn}[thm]{Definition}

\newtheorem{disc}[thm]{Discussion}

\newtheorem{rem}[thm]{Remark}
\newtheorem{exam}[thm]{Example}

\newtheorem{ques}[thm]{Question}

\theoremstyle{remark}
\newtheorem{remk}[thm]{Remark}

\newcounter{item}

\newenvironment{alist}[1][1]{\begin{list}
  {\textup{(\alph{item})}}{\usecounter{item} \setcounter{item}{#1}\addtocounter{item}{-1}
  \setlength{\itemsep}{0ex}
  \setlength{\topsep}{0ex} \setlength{\parsep}{0ex} \setlength{\labelwidth}{15mm}
  \setlength{\leftmargin}{10mm} } }{\end{list}}
\newenvironment{nlist}[1][1]{\begin{list}
  {\textup{(\arabic{item})}}{\usecounter{item} \setcounter{item}{#1}\addtocounter{item}{-1}
  \setlength{\itemsep}{0ex}
  \setlength{\topsep}{0ex} \setlength{\parsep}{0ex} \setlength{\labelwidth}{15mm}
  \setlength{\leftmargin}{10mm} } }{\end{list}}

\newcommand{\ds}{\displaystyle}

\newcommand{\bx}{{\bf x}}

\newcommand{\nn}{\mathbb{N}}

\newcommand{\zz}{\mathbb{Z}}
\newcommand{\cc}{\mathbb{C}}

\newcommand{\im}{\textup{Im}}

\newcommand{\incl}{\hookrightarrow}

\newcommand{\dirlim}{\varinjlim}
\newcommand{\ulim}{\textnormal{ulim}}
\newcommand{\Los}{\L{}os}
\newcommand{\cl}{\textnormal{cl}}
\newcommand{\regseq}{regular sequence}
\newcommand{\sop}{system of parameters}
\newcommand{\sops}{systems of parameters}
\newcommand{\charp}{characteristic $p>0$}

\newcommand{\echarz}{equal characteristic zero}
\newcommand{\CM}{Cohen-Macaulay}
\newcommand{\cal}{\mathcal}
\newcommand{\fg}{finitely-generated}

\begin{document}

\title[Big C-M and Seed Algebras \ldots\ Via Ultraproducts]{Big Cohen-Macaulay and Seed Algebras in Equal Characteristic Zero Via Ultraproducts}

\author{Geoffrey D. Dietz}
\address{Department of Mathematics,
Gannon University, Erie, PA 16541}
\email{gdietz@member.ams.org}

\author{Rebecca R.G.}
\address{Mathematics Department, Syracuse University, Syracuse, NY 13244}
\email{rirebhuh@syr.edu}
\thanks{The second author was partially supported by the National Science Foundation [grant numbers DGE 1256260, DMS 1401384].}

\date{August 30, 2016}


\begin{abstract}
Let $R$ be a commutative, local, Noetherian ring. In a past article, the first author developed a theory of $R$-algebras, termed seeds, that can be mapped to balanced big Cohen-Macaulay $R$-algebras. In prime characteristic $p$, seeds can be characterized based on the existence of certain colon-killers, integral extensions of seeds are seeds, tensor products of seeds are seeds, and the seed property is stable under base change between complete, local domains. As a result, there exist directed systems of big Cohen-Macaulay algebras over complete, local domains. In this work, we will show that these properties can be extended to analogous results in equal characteristic zero. The primary tool for the extension will be the notion of ultraproducts for commutative rings as developed by Schoutens and Aschenbrenner. 
\end{abstract}

\maketitle

For a local ring $(R,\mathfrak{m})$, a big Cohen-Macaulay $R$-algebra $B$ is an $R$-algebra such that some system of parameters for $R$ forms a regular sequence on $B$ with the extra property that $\mathfrak{m}B \neq B$ to ensure that $B$ is not trivial. If this is true for every system of parameters for $R$, then $B$ is a balanced big \CM\ algebra. Big Cohen-Macaulay algebras were first shown to exist in \cite{HH92} using characteristic $p$ methods related to tight closure theory (see \cite{HH90}). Hochster and Huneke proved that the absolute integral extension $R^+$ is a balanced big Cohen-Macaulay $R$-algebra when $R$ is an excellent, local, domain of prime characteristic. These results were extended in further articles. For example, in \cite{Ho94} Hochster makes explicit use of tight closure to provide an alternative proof of the existence of big Cohen-Macaulay algebras in prime characteristic based on the notion of algebra modifications. In \cite{HH95} Hochster and Huneke proved the existence of balanced big Cohen-Macaulay algebras for rings containing a field of characteristic zero and proved the ``weakly functorial'' existence of big Cohen-Macaulay algebras, i.e., given complete local domains of equal characteristic $R\to S$, there exists a balanced big Cohen-Macaulay $R$-algebra $B$ and a balanced big Cohen-Macaulay $S$-algebra $C$ such that $B\to C$ extends the map $R\to S$. The results in equal characteristic zero are based on reductions to prime characteristic that rely on Artin approximation.

In \cite{D07}, the first author introduced the notion of seed algebras. Given a local, Noetherian ring $R$, a seed algebra over $R$ is an $R$-algebra $S$ such that there exists an $R$-algebra map $S\to B$ where $B$ is a balanced big Cohen-Macaulay $R$-algebra. Some of the noteworthy results in that paper for rings of prime characteristic are Theorem~4.8 (which characterizes seeds based on durable, colon-killers), Theorem~6.9 (which shows that integral extensions of seeds are seeds and thus partially generalizes Hochster and Huneke's result for $R^+$), Theorem~7.8 (which shows that every seed can be mapped to an absolutely integrally closed, $\mathfrak{m}$-adically separated, quasilocal, balanced big Cohen-Macaulay algebra domain), Theorem~8.4 (which shows that tensor products of seeds are seeds and thus that any two big Cohen-Macaulay algebras map to a common big Cohen-Macaulay algebra), and Theorem~8.10 (which shows that the seed property is stable under base change and thus generalizes the weakly functorial existence result of Hochster and Huneke). 

Our aim in this paper is to extend the results of \cite{D07} to equal characteristic zero using the notion of ultraproducts of rings as developed by Schoutens and Aschenbrenner in works such as \cite{Sch03}, \cite{AS07}, and \cite{Sch10}, to name just a few. Although many reductions to characteristic $p$ have relied more directly upon various forms of Artin approximation, we decided that the theory of ultraproducts better suited the arguments here. Theorems 3.3, 3.5, 4.2, 4.3, and 5.1 in this work will serve as analogues to the above mentioned Theorems 8.4, 8.10, 4.8, 7.8, and 6.9 (respectively) from \cite{D07}. As an added bonus, we do not need a hypothesis of completeness on the rings in this work as was needed in \cite{D07}. The equal characteristic zero local domains $R$ in this work (complete or not) will each have an associated \textit{Lefschetz hull} $\mathscr{D}(R)$ which is an ultraproduct of rings $R_w$ \textit{almost all} of which are complete, local domains of prime characteristic. See Discussion~\ref{hypoth} and Theorem~\ref{transfer} for details.\ on the Lefschetz hull and its properties. We will then prove that properties of the $R_w$ transfer back to $R$, even without the hypothesis of completeness on $R$.

In the first section, we will review some properties of ultraproducts and big \CM\ algebras. With those tools established, we will proceed to discuss the notion of seeds in equal characteristic zero and then prove the analogues of the results named above from \cite{D07} in the following sections.

Throughout this paper, all rings are assumed to be commutative with identity, but not necessarily Noetherian unless stated explicitly. We will also follow the convention that ``quasilocal'' means that a ring (not necessarily Noetherian) has a unique maximal ideal while ``local'' means quasilocal and Noetherian.

\section{Overview of Ultraproducts and Big \CM\ Algebras}

We begin our efforts by reviewing background material necessary for our results. The main area of focus will be the topics related to ultraproducts in commutative algebra. The primary references will be the works of Schoutens and Aschenbrenner. See \cite{Sch10}, \cite{AS07}, and \cite{Sch03} for examples. Some of the results needed are listed below. References are provided to relevant sections of the works mentioned above.

\begin{defn}[(2.1.1), \cite{Sch10}]
Let $W$ be an infinite set. A \textit{filter} on $W$ is a proper, nonempty collection $\cal{W}$ of subsets of $W$ that is closed under finite intersections 
and taking supersets. If we add the conditions that the empty set is not in $\cal{W}$ and that for every subset $A$ of $W$, either $A \in \cal{W}$ or $W - A \in \cal{W}$, then $\cal{W}$ is an \textit{ultrafilter}. We say that $\cal{W}$ is \textit{principal} if it contains a least element, i.e., a set contained in every other element of $\cal{W}$. This least element in a principal ultrafilter will always be a singleton set. An ultrafilter that does not have a least element is \textit{non-principal}. Note that all of the sets in a non-principal ultrafilter are infinite, as intersections of finite sets always result in smaller sets and eventually lead to the empty set.

Following Schoutens, we will say (equivalently) that a \textit{non-principal ultrafilter} on $W$ is a collection $\cal{W}$ of infinite subsets of $W$ that is closed under finite intersection and such that for any $W' \subseteq W$, either $W' \in \cal{W}$ or $W-W' \in \cal{W}$. Note that the requirement that all sets in $\cal{W}$ be infinite implies that the empty set is not in $\cal{W}$, and that $\cal{W}$ is not principal as it cannot contain a singleton set. Also, if $W' \in \cal{W}$ and $W' \subseteq W''$, then under this notion we will have $W'' \in \cal{W}$ as otherwise $W-W'' \in \cal{W}$. But $W-W''$ is disjoint from $W'$, so this would imply that the empty set is in $\cal{W}$.

All ultrafilters in this work will be non-principal ultrafilters. 
\end{defn}


\begin{exam}
Let $W$ be any infinite set, and take the collection of subsets of $W$ whose complements are finite. This 
can be expanded to a non-principal ultrafilter via the axiom of choice. Hence there is always at least one ultrafilter on any infinite set $W$. See Section 6.2 of \cite{Ho93} for more details.
\end{exam}

\begin{rem}
In what follows, we will generally work with a fixed non-principal ultrafilter on a given set $W$. In the later sections, $W$ will be the set of positive prime integers, but the specific ultrafilter $\cal{W}$ on $W$ will not be described as it mostly does not play a determining role in the results. The presence of a non-principal ultrafilter is all that is needed.
\end{rem}

The purpose of the ultrafilter is to specify which subsets of $W$ are ``large" and thus develop a notion of when a property is true for ``almost all" elements in $W$. 

\begin{defn}
Let $W$ be an infinite set, and $\cal{W}$ a non-principal ultrafilter on $W$. We say that a property holds for \textit{almost all} $w \in W$ if it holds for all $w$ in some $W' \in \cal{W}$.
\end{defn}

In the later sections, we will prove a number of results in \echarz\ by appealing to a similar result already known to be true in \charp. The idea is that many rings of \echarz\ have the structure of an ultraproduct, defined below, of rings of \charp, and many properties of the factors transfer to their ultraproduct.

\begin{defn}[(2.1.2)--(2.1.4), \cite{Sch10}]
Let $W$ be an infinite set with a non-principal ultrafilter $\mathcal{W}$. For each $w \in W$, take a ring $A_w$. The \textit{ultraproduct} $A_\natural$ of the $A_w$ (with respect to  $\cal{W}$) is the quotient $\left(\Pi_w A_w\right)/I_{null}$, where $I_{null}$ is the ideal of elements $(x_w)_{w\in W}$ of $\Pi_w A_w$ where $x_w =0$ for almost all $w$. Any such ring $A_\natural$ is called an \textit{ultraring}.

Given an element $(a_w) \in \Pi_w A_w$, we refer to its image $a_\natural$ in $A_\natural$ as the ultraproduct of the $a_w$, and denote it by $\ulim_w\ a_w$.

We can take ultraproducts of $A_w$-modules as well: the ultraproduct of the $A_w$-modules $M_w$ is $M_\natural=\left(\Pi_w M_w\right)/M_{null}$, where $M_{null}$ consists of all elements of $\Pi_w M_w$ almost all of whose entries are 0. $M_\natural$ has the structure of an $A_\natural$-module and will be called an \textit{ultramodule}.
\end{defn}

Maps on the components of an ultraproduct lead to a map on the ultraproduct.

\begin{defn}[(2.1.7), \cite{Sch10}]
Given rings $A_w$ and $B_w$ and maps $f_w:A_w \to B_w$, the ultraproduct $f_\natural$ of the $f_w$ is defined to be the map $f_\natural(\ulim_w\ a_w)=\ulim_w\ f_w(a_w)$. We call $f_\natural$ a \textit{map of ultrarings}. 

Similarly, given $A_w$-modules $M_w$ and $N_w$ and maps $f_w:M_w \to N_w$, the ultraproduct $f_\natural$ of the $f_w$ is defined by $f_\natural(\ulim_w\ m_w)=\ulim_w\ f_w(m_w)$. We call $f_\natural$ a \textit{map of ultramodules}.
\end{defn}



The following theorem is one of the main tools for working with ultraproducts.

\begin{thm}[\Los's Theorem, (2.3.1), \cite{Sch10}]
Let $S$ be a system of equations
\[f_1=f_2=\ldots=f_s=0\] and inequalities
\[g_1 \ne 0,g_2 \ne 0, \ldots,g_t \ne 0\]
with $f_i,g_j \in \zz[x_1,\ldots,x_n]$. Then an element $a_\natural \in A_\natural$ is a solution of $S$ if and only if $a_w$ is a solution of $S$ in $A_w$ for almost all $w$.
\end{thm}

Ring theoretic properties that can be defined equationally can typically be transferred between an ultraring and almost all of its components via \Los' Theorem. For example, almost all $A_w$ are domains if and only if $A_\natural$ is a domain (\cite[2.4.10]{Sch10}). Other critical examples for us are given below.

\begin{exam}[(2.4.1)--(2.4.4), \cite{Sch10}]
A \textit{Lefschetz field} $K$ is a field of characteristic zero that is an ultraproduct almost all of whose components $K_w$ are fields of positive characteristic. The field of complex numbers $\cc$ is isomorphic to a Lefschetz field constructed from the algebraic closures of the finite fields of \charp. More generally, assuming the generalized Continuum Hypothesis, any uncountable algebraically closed field of characteristic zero is  also a Lefschetz field.
\end{exam}

\begin{exam}[(2.9.7), \cite{Sch03}]
Let $\bx_\natural$ be a sequence of elements in an ultraring $A_\natural$ with corresponding sequences of elements $\bx_w$ in each of the $A_w$. Then $\bx_\natural$ is a regular sequence on $A_\natural$ if and only if $\bx_w$ is a regular sequence on $A_w$ for almost all $w$.
\end{exam}

One of the most important constructions used below to transfer results from \charp\ to \echarz\ is the Lefschetz hull of an \echarz\ local ring. The reader is referred to \cite[Sections 4 and 5]{AS07} for the full treatment of these subjects, but we attempt to summarize some key results below. The following is the base assumption for all local rings in the following sections.

\begin{disc}\label{hypoth}
Let $R$ denote an \echarz\ Noetherian local ring that is not necessarily complete. Let $K$ denote a Lefschetz field with respect to an unspecified non-principal ultrafilter placed upon the set of positive prime integers. The components $K_w$ of $K$ will then be algebraically closed fields of characteristic $p(w)$.  The cardinality of $K$ needs to be sufficiently large, e.g., a cardinality exceeding $2^{|R|}$. 

One can add extra data (including a local homomorphism to $K$) so that there exists a \textit{Lefschetz hull} $\mathfrak{D}(R)$ which is defined to be an ultraproduct of rings $R_w$ (called the \textit{approximations of $R$}), where each $R_w$ is a complete local Noetherian ring with algebraically closed, \charp\ residue field $K_w$. The full definition and development of the Lefschetz hull is rather technical, and the curious reader is referred to \cite[Section 4]{AS07} for the complete details. We include several statements and theorems below that will attempt to describe the properties of the Lefschetz hull that are of the greatest interest to us for this present article.

By (\cite[(4.21)]{AS07}), there is a faithfully flat functorial map $R\to \mathfrak{D}(R)$. 
Given any $r\in R$, there is a corresponding element $\ulim_w r_w$ in $\mathfrak{D}(R)$, where each $r_w \in R_w$.
The $r_w$ are called \textit{approximations} of $r$. Also,  $\mathfrak{D}(R)$ is not necessarily Noetherian. 

Given a local homomorphism $R\to S$ of \echarz\ local rings, choose $K$ as above with sufficiently large cardinality compared to both rings. Again by (\cite[(4.26)]{AS07}), 
there exist faithfully flat Lefschetz hulls for both rings forming a commutative diagram
$$
\xymatrix{
\mathfrak{D}(R) \ar[r] & \mathfrak{D}(S) \\
R\ar[u]\ar[r] & S \ar[u]
}
$$

We will consider this setup as our standing hypotheses on local \echarz\ rings and on local homomorphisms between them in the later sections.
\end{disc}

The sense in which the components $R_w$ of the Lefschetz hull $\mathfrak{D}(R)$ are \charp\ approximations of $R$ is indicated by the following summary of results.

\begin{thm} \label{transfer}
Let $R$ be a local domain of \echarz\ with Lefschetz hull $\mathfrak{D}(R)$  and \charp\ approximations $R_w$ as in (\ref{hypoth}). 
\begin{enumerate}
\item \cite[(5.2.1)]{AS07} Almost all $R_w$ have the same dimension as $R$.
\item \cite[(5.4)]{AS07} A tuple $\bx$ from $R$ is a partial system of parameters for $R$ if and only if $\bx_w$ is a partial system of parameters in $R_w$ for almost all $w$. 
\item \cite[(5.17)]{AS07} If $R$ is a domain, then $\mathfrak{D}(R)$ is a domain and $R_w$ is a domain for almost all $w$. 
\end{enumerate}
\end{thm}

We close out this introduction with a very quick review of big Cohen-Macaulay algebras.

\begin{defn}
Let $(R,\mathfrak{m})$ be a local ring of dimension $d$. An $R$-algebra $B$ is a big \CM\ algebra over $R$ with respect to the system of parameters $x_1,\ldots,x_d$ for $R$ if $x_1,\ldots,x_d$ forms a regular sequence on $B$  and $\mathfrak{m}B \ne B$. We also call $B$ a big \CM\ algebra over $R$ if the particular \sop\ is not relevant.

An $R$-algebra $B$ is a balanced big \CM\ algebra over $R$ if every \sop\ on $R$ is a \regseq\ on $B$ and $\mathfrak{m}B \ne B$.
\end{defn}

As mentioned at the start of the paper, Hochster and Huneke \cite{HH92} first proved the existence of balanced big \CM\ algebras over excellent local rings $R$ of \charp\ by showing that every \sop\ from $R$ is a regular sequence on $R^+$. Hochster generalized this in \cite{Ho94} by showing how to construct a balanced big \CM\ algebra over $R$ using systems of algebra modifications. If $S$ is an $R$-algebra, and $sx_{k+1} = \sum_{i=1}^k x_is_i$ is a relation in $S$ on a partial \sop\ from $R$, then an algebra modification of $S$ is 
$$
T = \frac{S[U_1\ldots,U_k]}{s-\sum_{i=1}^k x_i U_i}
$$
For \charp\ local rings $R$, a direct limit of algebra modifications can be built into a balanced big \CM\ algebra over $R$.

Before we proceed further, we note that the \CM\ property is preserved by ultraproducts. The result below is a generalization of \cite[Theorem 6.4.7]{Sch10}, which proves something similar for the specific case of absolute integral closures $R_w^+$ and their ultraproducts.

\begin{lemma} \label{ultrabigCM}
Let $R$ be a local domain of \echarz\ with Lefschetz hull $\mathfrak{D}(R)$  and \charp\ approximations $R_w$ as in Discussion \ref{hypoth}. Suppose that $B_\natural  = \textnormal{ulim}_w\ B_w$, where $B_w$ is an $R_w$-algebra. Then 
\begin{alist}
\item $B_\natural$ is a big Cohen-Macaulay $R$-algebra with respect to a fixed \sop\ $\bx$ if and only if $B_w$ is a big Cohen-Macaulay $R_w$-algebra with respect to any choice of 
approximations $\bx_w$ for almost all $w$. 
\item $B_\natural$ is a balanced big Cohen-Macaulay $R$-algebra if and only if for almost all $w$, $B_w$ is a big Cohen-Macaulay algebra over $R_w$ with respect to the approximations $\bx_w$ for all \sops\ $\bx$ of $R$. 
\item If almost all $B_w$ are balanced big Cohen-Macaulay $R_w$-algebras, then $B_\natural$ is a balanced big Cohen-Macaulay $R$-algebra.
\end{alist}
\end{lemma}
\begin{proof}
Let $\bx$ be a sequence of elements from $R$ with corresponding approximations $\bx_w$ in each of the $R_w$. By \cite[Lemma 5.4]{AS07}, $\bx$ is a system of parameters for $R$ if and only if $\bx_w$ is a system of parameters for almost all of the $R_w$. By \Los' Theorem (or see the details of the proof for \cite[Lemma 5.4]{AS07}), $\bx$ will be a regular sequence on $B_\natural$ if and only if $\bx_w$ is regular on $B_w$ for almost all $w$.  Thus, $B_\natural$ is big Cohen-Macaulay over $R$ with respect to $\bx$ if and only if for almost all $w$, $B_w$ is big Cohen-Macaulay over $R_w$ with respect to $\bx_w$. Part (b) follows from part (a). The single direction for part (c) then follows from part (b) as any approximation of a \sop\ from $R$ will be a \sop\ for almost all $R_w$.
\end{proof}

The converse of part (c) is likely to be false, but we do not have an explicit counterexample. There is no \textit{a priori} reason why a big Cohen-Macaulay $R_w$-algebra with respect to approximations $\bx_w$ of the \sops\ of $R$ should still be big Cohen-Macaulay over $R_w$ with respect to \textit{all} \sops\ of $R_w$, i.e., not every \sop\ for $R_w$ will necessarily arise as an approximation of a \sop\ from $R$.

\section{Seeds in Equal Characteristic Zero}

\begin{defn}
Given a local ring $(R,\mathfrak{m})$, a seed $S$ over $R$ is an $R$-algebra $S$ such that there exists an $R$-algebra map $S\to B$ where $B$ is a (balanced) big Cohen-Macaulay $R$-algebra \cite{D07}. 
\end{defn}

\begin{rem}
By the result of Bartijn and Strooker \cite{BS83}, there is no harm in leaving off the adjective ``balanced" as any big Cohen-Macaulay algebra (or module) can be modified into a balanced version by applying the $\mathfrak{m}$-adic separated completion. 
\end{rem}

In order to apply a reduction to characteristic $p$ method via ultraproducts, we find it necessary to set forth the following modified definition for rings of \echarz.

\begin{defn}
Let $R$ be a local domain of \echarz\ with Lefschetz hull $\mathfrak{D}(R)$  and \charp\ approximations $R_w$ as in Discussion \ref{hypoth}.
An $R$-algebra $S$ is called a \textit{rational seed} if there exists an $R$-algebra map $S\to T_\natural$ such that $T_\natural = \textnormal{ulim}_w\ T_w$, where $T_w$ is a seed algebra over $R_w$ for almost all $w$. We will call the ultrarings $T_\natural$ \textit{ultraseeds}.
\end{defn}

Rational seeds exist in equal characteristic zero.

\begin{thm}
Let $R$ be a local domain of \echarz\ with Lefschetz hull $\mathfrak{D}(R)$  and \charp\ approximations $R_w$ as in Discussion \ref{hypoth}. Then $R$ is a rational seed and $\mathfrak{D}(R)$ is an ultraseed over $R$.
\end{thm}
\begin{proof}
Consider a Lefschetz hull $\mathfrak{D}(R)$ and the approximations $R_w$ in characteristic $p$. By \cite[(4.2) and (5.17)]{AS07} almost all of the $R_w$ are complete, local domains of characteristic $p$ and so can be mapped to balanced big Cohen-Macaulay $R_w$-algebras $B_w$ by \cite[(5.15)]{HH92} or \cite[(6.11)]{D07}. (For the cases where $R_w$ is not a complete, local domain, set $B_w$ to be the zero ring.) As almost all of the $R_w$ rings are seeds over themselves, $\mathfrak{D}(R)$ is an ultraseed over $R$. The map $R\to \mathfrak{D}(R)$ shows that $R$ is a rational seed.
\end{proof}

The following lemma provides alternative characterizations of rational seeds. In short, the lemma says that when working with a rational seed $S$, we can skip over the ultraseed $T_\natural$ and work directly with an ultraring $B_\natural$ that is also a (balanced) big Cohen-Macaulay $R$-algebra.

\begin{lemma} \label{ultraseed}
Let $R$ be a local domain of \echarz\ with Lefschetz hull $\mathfrak{D}(R)$  and \charp\ approximations $R_w$ as in Discussion \ref{hypoth}.  Then the following are equivalent for an $R$-algebra $S$:
\begin{nlist}
\item $S$ is a rational seed over $R$.
\item There exists an $R$-algebra map $S\to B_\natural$ such that $B_\natural$ is an ultraring and a balanced big Cohen-Macaulay $R$-algebra.
\item There exists an $R$-algebra map $S\to B_\natural$ such that $B_\natural$ is an ultraring and a big Cohen-Macaulay $R$-algebra with respect to a fixed \sop\ from $R$.
\end{nlist}
\end{lemma}
\begin{proof}
For (1) $\Rightarrow$ (2), there is a map $S\to T_\natural$, where almost all approximations $T_w$ are seeds over $R_w$. Thus, almost all of the $T_w$ can be mapped to balanced big Cohen-Macaulay $R_w$-algebras $B_w$. Let $B_\natural$ be the ultraproduct of the $B_w$. Then the map $S\to T_\natural$ and maps $T_w\to B_w$ induce a composite map $S\to B_\natural$. Moreover, Lemma~\ref{ultrabigCM}(c) implies that $B_\natural$ is a balanced big Cohen-Macaulay $R$-algebra. The implication (2) $\Rightarrow$ (3) is obvious.
Finally, for (3) $\Rightarrow$ (1) note that Lemma~\ref{ultrabigCM}(a) implies that almost all of the approximations $B_w$ are big Cohen-Macaulay $R_w$-algebras for some fixed \sop\ from $R_w$ and thus are seeds over $R_w$.
\end{proof}

The next result indicates that $R$-algebras that are formed by algebra modifications of a rational seed 
are also rational seeds.

\begin{prop}
Let $R$ be a local domain of \echarz\ with Lefschetz hull $\mathfrak{D}(R)$  and \charp\ approximations $R_w$ as in Discussion \ref{hypoth}. Let $S$ be a rational seed over $R$. Suppose that $T$ is an $S$-algebra that is also a direct limit of algebra modifications of $S$ with respect to relations on the \sops\ of $R$. Then $T$ is a rational seed.
\end{prop}
\begin{proof}
As $T$ is a direct limit of algebra modifications of $S$, it maps to every $S$-algebra that is also a balanced big \CM\ $R$-algebra \cite{RG15}. Since $S$ is a rational seed, Lemma~\ref{ultraseed} implies that there is a map $S\to B_\natural$, a balanced big \CM\ $R$-algebra and an ultraring. Thus, there is a map $T\to B_\natural$, showing that $T$ is a rational seed.\end{proof}

\begin{ques}
Are there $R$-algebra seeds $S$ in equal characteristic zero that are not rational seeds? 

This question reduces to the question of whether there are big Cohen-Macaulay $R$-algebras in equal characteristic zero that do not arise from some reduction to characteristic $p$ argument or from a sequence of algebra modifications but exist ``intrinsically" over $R$. Answering this question one way or another would inform us about whether or not there is a difference between rational seeds and ordinary seeds. For the time being however, the notion of rational seed seems robust enough.
\end{ques}

We can also show that direct limits of ultraseeds are ultraseeds under certain conditions, paralleling \cite[Lemma 3.2]{D07}. To do so we need a lemma about direct limits and ultralimits. 
Unlike the analogous work in \cite{D07}, Lemma~\ref{directultralimit} and Proposition~\ref{limultraseed} are not used in  deriving the remaining results in this article. 

\begin{lemma} \label{directultralimit}
Let $\Lambda$ be a directed set. Let $\mathcal{W}$ be an ultrafilter on an infinite set $W$. Suppose that \emph{for all} $w\in W$ and all $\lambda\leq \mu \leq \nu \in \Lambda$ we have ring maps 
$$
\begin{array}{c}
\ds f_w^{(\lambda,\mu)}: A_w^{(\lambda)} \to A_w^{(\mu)}\\[2mm]
\ds f_w^{(\mu,\nu)}: A_w^{(\mu)} \to A_w^{(\nu)}\\[2mm]
f_w^{(\lambda,\nu)}: A_w^{(\lambda)} \to A_w^{(\nu)}
\end{array}
$$
such that  $f_w^{(\lambda,\nu)} = f_w^{(\mu,\nu)}\circ f_w^{(\lambda,\mu)}$. For all $\lambda \in \Lambda$ let $A_\natural^{(\lambda)} = \ulim\ A_w^{(\lambda)}$, for all $w\in W$ let $A_w = \dirlim_\lambda A_w^{(\lambda)}$, and let $A_\natural = \ulim\ A_w$. Then 
$$
A_\natural = \dirlim_\lambda A_\natural^{(\lambda)}
$$
in the category of ultrarings with respect to $\mathcal{W}$, with morphisms maps of ultrarings.
\end{lemma}
\begin{proof}
First note that as the maps $f_w^{(\lambda,\mu)}$, etc. exist for all $w\in W$ (and not just almost all $w$), the hypotheses on compositions of those maps induce analogous properties on maps $f^{(\lambda,\mu)}$ in the directed system of ultrarings $\{A_\natural^{(\lambda)}\}_{\lambda \in \Lambda}$ by \cite[(2.1.7)]{Sch10}. Therefore, $\dirlim_\lambda A_\natural^{(\lambda)}$ exists in the category of ultrarings. 

We will now prove that $A_\natural = \dirlim_\lambda A_\natural^{(\lambda)}$ by showing $A_\natural$ satisfies the universal mapping property of a direct limit in the category of ultrarings. Indeed, suppose that $C_\natural = \ulim\ C_w$ is an ultraring and that for all $\lambda\in \Lambda$ there exists a map of ultrarings $\psi^{(\lambda)}: A_\natural^{(\lambda)}\to C_\natural$ such that $\psi^{(\lambda)} = \psi^{(\mu)} \circ f^{(\lambda,\mu)}$ for all $\lambda\leq \mu$ in $\Lambda$. Therefore, for all $w\in W$ we have maps
$$
\psi_w^{(\lambda)}: A_w^{(\lambda)}\to C_w \mbox{ such that } \psi_w^{(\lambda)} = \psi_w^{(\mu)} \circ f_w^{(\lambda,\mu)}
$$
for all $\lambda \leq \mu$ in $\Lambda$. For each $w\in W$ we apply the universal mapping property of direct limits to construct a map $\psi_w: A_w \to C_w$ as $A_w =  \dirlim_\lambda A_w^{(\lambda)}$. As we have maps for all $w\in W$, \cite[(2.1.7)]{Sch10} implies that we have a map $\psi: A_\natural \to C_\natural$, which shows that $A_\natural$ is  $\dirlim_\lambda A_\natural^{(\lambda)}$ in the category of ultrarings. 
\end{proof}

\begin{prop} \label{limultraseed}
Let $R$ be a local domain of \echarz\ with Lefschetz hull $\mathfrak{D}(R)$  and \charp\ approximations $R_w$ as in Discussion \ref{hypoth}. Adopt the same notation and hypotheses of Lemma~\ref{directultralimit} for $R_w$-algebras $A_w^{(\lambda)}$ and $A_w$ and ultrarings $A_\natural^{(\lambda)}$ and $A_\natural = \ulim\ A_w = \dirlim_\lambda A_\natural^{(\lambda)}$.
Then $A_\natural$ is an ultraseed (i.e., $A_w$ is a seed over $R_w$ for almost all $w$) if and only if $P = \{w\, |\, A_w^{(\lambda)} \textnormal{ is a seed over } R_w \textnormal{ for all } \lambda \}$ is a set in the ultrafilter $\mathcal{W}$ on $W$.
\end{prop}
\begin{proof}
For the forward direction, note that if $A_\natural$ is an ultraseed, then the maps $A_\natural^{(\lambda)}\to A_\natural$ in the direct limit system arise from maps $A_w^{(\lambda)}\to A_w$ for all $w$, using the previous lemma to identify $A_\natural$ with $\ulim\ A_w$. As the $A_w$ are seeds over $R_w$ for almost all $w$, for these values of $w$, the maps $A_w^{(\lambda)}\to A_w$ imply that $A_w^{(\lambda)}$ is a seed over $R_w$ for all $\lambda$, showing that the set $P$ is in the ultrafilter $\mathcal{W}$. 

Conversely, suppose that the set $P$ is in the ultrafilter $\mathcal{W}$. Then for almost all $w$ we have that $A_w^{(\lambda)}$ is a seed over $R_w$ for all $\lambda$. Hence for almost all $w$ we have $A_w = \dirlim_\lambda A_w^{(\lambda)}$ is a seed over $R_w$ \cite[Lemma 3.2]{D07}, which implies that $A_\natural = \ulim\ A_w$ is an ultraseed. 
\end{proof}

\begin{remk}
The proposition above, unlike most other results in this work, depends upon the choice of ultrafilter $\mathcal{W}$. The subtlety lies in the fact that the set $P$ in the statement is equal to the \textit{possibly infinite} intersection of the sets $P_\lambda = \{ w \,|\, A_w^{(\lambda)} \textnormal{ is a seed over } R_w\}$ . While each $A_\natural^{(\lambda)}$ is an ultraseed if and only if $P_\lambda$ is in the ultrafilter, the set $P$ may or may not still live in the ultrafilter as ultrafilters need only be closed under finite intersections. There is no guarantee for infinite intersections.
\end{remk}


\section{Tensor Products and Base Change}

In this section we aim to prove the analogues of two major results from \cite{D07}: tensor products of seeds are seeds and base change applied to a seed maintains the seed property. As both results relate to tensor products, we need an initial elementary lemma relating tensor products to ultraproducts. The result is a partial generalization of \cite[Proposition 1.2]{AS07} for tensor products of arbitrary modules.

\begin{lemma} \label{ultratensor}
Let $A_\natural$ be an arbitrary ultraring, i.e., an ultraproduct of (not necessarily Noetherian) commutative rings $A_w$. Let $M_w$ and $N_w$ be arbitrary (not necessarily finitely generated) $A_w$-modules.
Then there exists a canonical map $M_\natural \otimes_{A_\natural} N_\natural \to \ulim (M_w \otimes_{A_w} N_w)$.
\end{lemma}
\begin{proof}
Let $m_\natural, m'_\natural \in M_\natural$ and $n_\natural \in N_\natural$ with approximations $m_w$, $m'_w$,  and $n_w$, respectively. Define a map $M_\natural \times N_\natural \to (M_w \otimes_{A_w} N_w)_\natural$ by sending 
$$
(m_\natural,n_\natural) \mapsto \textnormal{ulim} (m_w\otimes n_w). 
$$
This map is well-defined on equivalence classes for $m_\natural$ or $n_\natural$. Indeed, if $m_\natural = 0$, then almost all $m_w = 0$ and so  $(m_w\otimes n_w)=0$ for almost all $w$ so that the ultraproduct is also $0$. A similar situation holds when $n_\natural = 0$. The map is also bilinear: we have 
$$
\begin{array}{rcl}
(m_\natural + m'_\natural,n_\natural) = ((m+m')_\natural,n_\natural) & \mapsto & \textnormal{ulim}((m_w+m'_w)\otimes n_w) \\
& = & \textnormal{ulim}(m_w\otimes n_w) + \textnormal{ulim}(m'_w\otimes n_w).
\end{array}
$$
Similarly, the map is additive in the second coordinate. Finally, if $a_\natural \in A_\natural$ with approximations $a_w$, then
$$
(a_\natural m_\natural,n_\natural) = ((a_w m_w)_\natural,n_\natural) \mapsto \textnormal{ulim}(a_w m_w \otimes n_w) = a_\natural \cdot  \textnormal{ulim}(m_w\otimes n_w).
$$
By the universal mapping property of tensor products, we now have the desired map $M_\natural \otimes_{A_\natural} N_\natural \to \ulim (M_w \otimes_{A_w} N_w)$.
\end{proof}

\begin{lemma}\label{ultratensorgen}
Let $A_\natural$ be an arbitrary ultraring, i.e., an ultraproduct of (not necessarily Noetherian) commutative rings $A_w$. Let $M_w^{(i)}$ be arbitrary (not necessarily finitely generated) $A_w$-modules indexed by a common set $\mathcal{I}$.
Then there exists a canonical map $\bigotimes_{i\in \mathcal{I}} M_\natural^{(i)} \to \ulim\left( \bigotimes_{i\in \mathcal{I}} M_w^{(i)}\right)$. [Note that the first tensor product has $A_\natural$ as its base ring while the second has the $A_w$ rings as bases.]
\end{lemma}
\begin{proof}
A similar argument to that above can be used to show that a similar map defined on the arbitrary product $\Pi_{i\in \mathcal{I}} M_\natural^{(i)}$ is $A_\natural$-linear in every coordinate, which then gives a map from the tensor product to the desired ultraproduct of tensor products.
\end{proof}

We present one of our main theorems next, which shows that tensor products of rational seeds are also rational seeds, an analogue of \cite[Theorem 8.4]{D07}.

\begin{thm} \label{tensorseeds}
Let $R$ be a local domain of \echarz\ with Lefschetz hull $\mathfrak{D}(R)$  and \charp\ approximations $R_w$ as in Discussion \ref{hypoth}. Let $(S_i)_{i\in \mathcal{I}}$ be a finite family of rational seeds over $R$. Then $\bigotimes_{i\in \mathcal{I}} S_i$ (tensored over $R$) is also a rational seed. Consequently, if $B_\natural$ and $B'_\natural$ are (balanced) big \CM\ $R$-algebras and ultrarings, then there exists a (balanced) big \CM\ $R$-algebra and ultraring $C_\natural$ filling in the commutative diagram:
$$
\xymatrix{B_\natural \ar[r] & C_\natural \\ R\ar[r]\ar[u] & B'_\natural\ar[u]}
$$
\end{thm}
\begin{proof}
By the definition of rational seed, for each $i\in \mathcal{I}$ we have a map $S_i \to T_\natural^{(i)}$ such that $T_\natural^{(i)}$ is an ultraproduct of $R_w$-algebras, almost all of which are seeds over $R_w$. 
As the set $\mathcal{I}$ is finite and a finite intersection of sets in an ultrafilter still lies in the ultrafilter, we can say that almost all $T_w^{(i)}$ are $R_w$-seeds across all $i\in \mathcal{I}$ simultaneously. (In other words, the ``almost all $w$" is independent of the choice of $i$.) Thus for almost all $w$, we have that $\bigotimes_{i\in \mathcal{I}} T_w^{(i)}$ (tensored over $R_w$) is a seed over $R_w$ by \cite[Theorem 8.4]{D07}. Finally note that we have maps
$$
\bigotimes_{i\in \mathcal{I}} S_i\to \bigotimes_{i\in \mathcal{I}} T_\natural^{(i)} \to \ulim\left( \bigotimes_{i\in \mathcal{I}} T_w^{(i)} \right),
$$
where the first product uses $R$ as base, the second uses $\mathfrak{D}(R)$ as base, the third uses $R_w$ as base, and the last map follows from Lemma~\ref{ultratensorgen}. This shows that $\bigotimes_{i\in \mathcal{I}} S_i$ is a rational seed.

Now that we have established that tensor products of rational seeds are rational seeds, we can deduce the final claim by noting that $B_\natural$ and $B'_\natural$ are rational seeds by Lemma~\ref{ultraseed} and Lemma~\ref{ultrabigCM}. As a result, there is a map 
$$
B_\natural\otimes_R B'_\natural\to C_\natural
$$
as the first part of the proof shows that $B_\natural\otimes_R B'_\natural$ is a rational seed, and Lemma~\ref{ultrabigCM} shows that there is a map to a balanced big \CM\ $R$-algebra $C_\natural$ that is also an ultraring.
\end{proof}

\begin{remk}
The \charp\ version of the theorem above is true for arbitrarily large families of tensor products. We kept the \echarz\ version limited to finite families because of subtleties related to ultrafilters. One can conclude that an infinite tensor product of rational seeds is a rational seed as long as the set 
$$
\{ w \,|\, T_w^{(i)} \textnormal{ is a seed over } R_w \textnormal{ for all } i \in \mathcal{I} \}
$$
is a member of the ultrafilter. This is similar to the restriction applied in Proposition~\ref{limultraseed} earlier.
\end{remk}

We are also able to derive an \echarz\ version of \cite[Theorem 8.10]{D07}, which will show that the rational seed property is stable under base change between local domains of \echarz.

\begin{thm}\label{basechange} 
Let $R$ and $S$ be local domains of \echarz\ with a local map $R\to S$ and Lefschetz hulls $\mathfrak{D}(R)$ and $\mathfrak{D}(S)$ with \charp\ approximations $R_w$ and $S_w$ as in Discussion \ref{hypoth}. Suppose that $T$ is a rational seed over $R$. Then $T\otimes_R S$ is a rational seed over $S$. Consequently, if $B_\natural$ is an ultraproduct and a big \CM\ $R$-algebra, then there exists $C_\natural$ which is a balanced big \CM\ algebra over $S$ and an ultraproduct that fills in a commutative diagram:
$$
\xymatrix{B_\natural \ar[r] & C_\natural \\ R\ar[r]\ar[u] & S\ar[u]}
$$
\end{thm}
\begin{proof}
As $T$ is a rational seed over $R$, there exists an $R$-linear map $T \to U_\natural$, where almost all $U_w$ are seeds over the complete, local domains $R_w$ of \charp. By \cite[Theorem 8.10]{D07}, we know that $U_w\otimes_{R_w} S_w$ is a seed over $S_w$ for almost all $w$. Therefore $\ulim(U_w\otimes_{R_w} S_w)$ is an ultraseed over $S$. Note that we have $S$-linear maps:
$$
T\otimes_R S \to U_\natural \otimes_R S \to U_\natural \otimes_{\mathfrak{D}(R)} \mathfrak{D}(S) \to \ulim(U_w\otimes_{R_w} S_w),
$$
where the last map follows from Lemma~\ref{ultratensor}. This chain of $S$-algebra maps shows that $T\otimes_R S$ is a rational seed over $S$.

In the special case that $T=B_\natural$ is a big \CM\ $R$-algebra, $B_\natural \otimes_R S$ will map further to an $S$-algebra $C_\natural$ with the desired properties by Lemma~\ref{ultraseed}.
\end{proof}

\section{Other Properties of Seeds}

In this section we present two primary results, both of which are \echarz\ analogues of results from \cite{D07}. The first result will characterize rational seeds in terms of ``durable colon-killers" while the second will show that rational seeds can be mapped to balanced big \CM\ algebras with a host of other nice properties.

We start by restating the definition of durable colon-killer from \cite[Definition 4.7]{D07}.

\begin{defn} 
For a local ring $(R,\mathfrak{m})$ and an $R$-algebra $S$,
an element $c\in S$ is called a \textit{weak durable colon-killer} over $R$
if for some system of parameters $x_1,\ldots,x_n$ of $R$,
$$c( (x_1^{t},\ldots,x_k^{t})S:_S x_{k+1}^{t} ) \subseteq
(x_1^{t},\ldots,x_k^{t})S,$$
for all $1\leq k\leq n-1$ and all $t\in\nn$, and if for each 
$N\geq 1$, there exists a $k \ge 1$ such that 
$c^N\not\in \bigcap_k \mathfrak{m}^k S$. An element $c\in
S$ will be called a \textit{durable colon-killer} over $R$
if it is a weak durable colon-killer for every system of parameters of $R$.
\end{defn}

Notice that if $S=R$, then all colon-killers in $R$ that are not
nilpotent are durable colon-killers, and if $R$ is reduced, all
nonzero colon-killers are durable colon-killers. The following is an analogue of \cite[Theorem 4.8]{D07}.

\begin{thm}
Let $(R,\mathfrak{m})$ be a local domain of \echarz\ with Lefschetz hull $\mathfrak{D}(R)$  and \charp\ approximations $R_w$ as in Discussion \ref{hypoth}. Then the following three properties are equivalent for an $R$-algebra $S$:
\begin{nlist}
\item $S$ is a rational seed.
\item There exists a map $S\to T_\natural$ such that $T_\natural$ contains a durable colon-killer $c$.
\item There exists a map $S\to T_\natural$ with element $c\in T_\natural$ such that $c = \ulim\ c_w$ where $c_w$ is a weak durable colon-killer in almost all $T_w$.
\end{nlist}
Moreover, property (4) below implies (1)--(3):
\begin{nlist}
\setcounter{item}{3}
\item There exists a map $S\to T_\natural$ with element $c\in T_\natural$ such that $c = \ulim\ c_w$ where $c_w$ is a durable colon-killer in almost all $T_w$.
\end{nlist}
\end{thm}
\begin{proof}
(1) along with Lemma~\ref{ultraseed} imply that there exists $S\to B_\natural$ such that $B_\natural$ is a balanced big \CM\ $R$-algebra. Thus, (1) implies (2) using $T_\natural = B_\natural$ and $c=1$.

(2) implies (3) when one takes approximations $c_w$ of $c$ and applies \L{}os' Theorem. Indeed, fix a \sop\ $x_1, \ldots, x_d$ for $R$ and take approximations $x_{1w}, \ldots, x_{dw}$ in each $R_w$ via $\mathfrak{D}(R)$. Then almost all sequences $x_{1w},\ldots,x_{dw}$ form a \sop\ in $R_w$ by \cite[Corollary 4.3.8]{Sch10}. Powers of the \sop\ and their approximations are also \sops\ for $R$ and almost all $R_w$, respectively. By \Los' Theorem and the fact that $c$ is a colon-killer, almost all $c_w$ will be weak colon-killers with respect to the powers of the approximations of the \sop\ $\bx$. They will be durable as $c^N \not\in \bigcap_k \mathfrak{m}^k T_\natural$ for all $N$ implies that the same is true for almost all of the approximations $c_w^N$ by \Los' Theorem. Indeed, suppose that for some $N$ we have $c_w^N \in \bigcap_k \mathfrak{m}_w^k T_w$ for almost all $w$. As $\mathfrak{m}$ is a finitely generated ideal, each inclusion can be specified equationally using polynomials and thus \Los' Theorem would imply that the same inclusion will hold for $c^N$, a contradiction. 

We will show that (3) and (4) each imply (1) to complete the proof. By \cite[Theorem 4.8]{D07}, (3) and (4) each imply that almost all of the $T_w$ are seeds over $R_w$. Thus, the map $S\to T_\natural$ shows that $S$ is a rational seed.
\end{proof}

We also present an analogue of \cite[Theorem 7.8]{D07} showing that rational seeds map to ultrarings with a host of other nice properties.

\begin{thm}\label{ultranice}
Let $(R,\mathfrak{m})$ be a local domain of \echarz\ with Lefschetz hull $\mathfrak{D}(R)$  and \charp\ approximations $R_w$ as in Discussion \ref{hypoth}. If $S$ is a rational seed over $R$, then $S$ maps to a quasilocal, absolutely integrally closed, $\mathfrak{m}$-adically separated, balanced big \CM\ $R$-algebra domain $B_\natural$.
\end{thm}
\begin{proof}
There exists a map $S\to T_\natural$ such that almost all $T_w$ are seeds over $R_w$. By \cite[Theorem 7.8]{D07}, almost all of the $T_w$ can be mapped to quasilocal, absolutely integrally closed, $\mathfrak{m}$-adically separated, balanced big \CM\ $R_w$-algebra domains $B_w$. (Set $B_w$ to be the zero ring for the cases where $T_w$ is not a seed so that there is still a map $T_w\to B_w$.) By Lemma~\ref{ultrabigCM}(c), we have that $B_\natural$ is a balanced big \CM\ $R$-algebra with a composite map $S\to T_\natural \to B_\natural$ induced by the maps $T_w\to B_w$. Note that \cite[(2.4.10)]{Sch10} implies $B_\natural$ is a domain as almost all $B_w$ are domains, and \cite[(2.4.8)]{Sch10} implies that $B_\natural$ will be quasilocal. 

To see that $B_\natural$ is also absolutely integrally closed, note that as the $B_w$ are almost all absolutely integrally closed, all monic polynomials split in those $B_w$. Consider a monic polynomial in $B_\natural$. Then almost all approximations over $B_w$ will split. By \Los' Theorem, the original polynomial will split over $B_\natural$ as well. 

Finally, to see that $B_\natural$ is $\mathfrak{m}$-adically separated,  let $b_\natural \in \bigcap_k \mathfrak{m}^k B_\natural$. Via the Lefschetz hull $\mathfrak{D}(R)$ and \Los' Theorem we obtain approximations $b_w \in \bigcap_k \mathfrak{m}_w^k B_w$ for almost all $w$. Due to the separation for almost all $w$, we have that $b_w=0$ for almost all $w$. Thus, $b=0$ and so $B_\natural$ is $\mathfrak{m}$-adically separated.
\end{proof}

\section{Integral Extensions of Seeds}

We now show the final promised result: integral extensions of rational seeds are also rational seeds, which is an analogue of 
\cite[Theorem 6.9]{D07}.

\begin{thm}
Let $R$ be a local domain of \echarz\ with Lefschetz hull $\mathfrak{D}(R)$  and \charp\ approximations $R_w$ as in Discussion \ref{hypoth}. Let $S$ be a rational seed over $R$, and let $T$ be an integral extension of $S$. Then $T$ is also a rational seed.
\end{thm}
\begin{proof}
Start with the integral extension $S\incl T$. We will expand this map into the diagram below, where $I$ will be an ideal of $S$ and $U_\natural$ will be an ultraseed domain. We will show that all vertical maps are integral extensions.
$$
\xymatrix{
S\ar@{^{(}->}[d]\ar[r] & S/I\ar@{^{(}->}[d] \ar@{^{(}->}[r] & U_\natural\ar@{^{(}->}[d] \\
T\ar[r] & T/IT \ar[r] & T/IT \otimes_{S/I}  U_\natural
}
$$
Indeed, by Theorem~\ref{ultranice} the rational seed $S$ maps to an ultraseed domain $U_\natural$. (It actually has far more properties, but for this argument we only need it to be an ultraseed and a domain.) 
Let $S/I$ be the homomorphic image of $S$ within the domain $U_\natural$, which implies that $S/I$ is a domain and so $I$ is a prime ideal of $S$. As $S\incl T$ is integral, the induced map $S/I \to T/IT$ will still be integral. As $I$ is a prime ideal (and hence integrally closed) and $T$ is an integral extension of $S$, we have $IT\cap S = I$ and so $S/I \to T/IT$ remains injective as well.
The existence of the two maps to $T/IT \otimes_{S/I}  U_\natural$ are then clearly established, but we need to show that the rightmost vertical map is an integral extension. 
The map is integral as it is a base change of the integral map $S/I\to T/IT$ (or $S\to T$). 
For injectivity, we appeal to the general lemma \cite[Lemma 6.2]{D07}, which proves that given the integral extension $S/I\to T/IT$ and the domain (reduced is sufficient) extension $S/I \to U_\natural$, the vertical map on the right is also injective.

In order to prove the theorem, it now suffices to prove that the ring $T/IT \otimes_{S/I}  U_\natural$ is a rational seed. Thus, we may reduce the whole problem to the case that $S=U_\natural$ is an ultraseed domain, and $T= (U_\natural)^+$, the absolute integral closure of $U_\natural$ within an algebraic closure $L_\natural$ of the fraction field of $U_\natural$. By \cite[Remark 2.4.4]{Sch10}, $L_\natural$ is also Lefschetz.

For almost all $w$, we have that $U_w$ is a domain and a seed over $R_w$, a complete local domain of \charp, and $L_w$ is an algebraic closure of the fraction field of $U_w$. Thus, for almost all $w$ we have the inclusions:
$$
U_w \incl (U_w)^+ \incl L_w.
$$
By \cite[Theorem 6.9]{D07}, almost all $(U_w)^+$ will be seeds over $R_w$ as they are integral extensions of the seeds $U_w$. Therefore, $V = \ulim((U_w)^+)$ is an ultraseed over $R$. To finish the proof, we claim that $V = (U_\natural)^+$ within $L_\natural$. 

Indeed, let $v_\natural\in V$. Then almost all approximations $v_w$ live in $(U_w)^+$ and satisfy a monic polynomial over $U_w$. Taking the ultralimit of the polynomials and applying \Los' Theorem shows that $v_\natural$ is the root of a monic polynomial over $U_\natural$ and so lives in $(U_\natural)^+$.
Conversely, let $t\in (U_\natural)^+$. As $t\in(U_\natural)^+\subseteq  L_\natural$, we have $t = \ulim\ t_w$ where  $t_w\in L_w$. Also, $t$ is the root of a monic polynomial over $U_\natural$, so another application of \Los' Theorem shows that almost all $t_w$ are roots of the corresponding approximated monic polynomials over $U_w$. Therefore, almost all $t_w$ live in $(U_w)^+$ within $L_w$, and so $t$ lives in $V=\ulim((U_w)^+)$ within $L_\natural$. 
\end{proof}

We can then use this result to deduce an \echarz\ result about the absolute integral closure $R^+$ of $R$, independent of Hochster and Huneke's result that $R^+$ is a balanced big \CM\ algebra in characteristic $p$. 

\begin{cor}
Let $R$ be a local domain of \echarz. Then the absolute integral closure $R^+$ of $R$ is a rational seed over $R$.
\end{cor}

\section{Applications to tight closure in \echarz}

Using Theorems~\ref{tensorseeds} and \ref{basechange}, we can define a closure operation on \echarz\ rings using balanced big Cohen-Macaulay algebras that are ultrarings formed from balanced big \CM\ $R_w$-algebras. 

\begin{defn}
\label{ourclosureop}
Let $R$ be a local domain of \echarz\  with Lefschetz hull $\mathfrak{D}(R)$  and \charp\ approximations $R_w$ as in Discussion \ref{hypoth}. Let $N\subseteq M$ be $R$-modules. For $u\in M$, we say that $u\in N_M^\cl$ if $1\otimes u \in \im(B_\natural\otimes_R N\to B_\natural\otimes_R M)$ for some balanced big \CM\ $R$-algebra $B_\natural$ that is an ultraproduct of balanced big \CM\ $R_w$-algebras. 
\end{defn}

As tight closure in \charp\ is equivalent to extension and contraction from some balanced big \CM\ algebra (\cite[Theorem 11.1]{Ho94}), the operation above provides a way to transfer tight closure from \charp\ to \echarz. In \cite[Definition 5.2]{Sch03}, Schoutens defines \textit{generic tight closure}, which is similar to the operation above as both rely on the approximations of an element residing in \charp\ tight closure for almost all $p$. However, they may not be equal. We do have the following:

\begin{prop}
Let $R$ be a local domain of \echarz\ with \charp\ approximations $R_w$ as in Discussion \ref{hypoth}. We can extend generic tight closure to \fg\ $R$-modules as in \cite[Lemma 7.0.5]{E12}. Then the closure operation cl given in Definition \ref{ourclosureop} is contained in generic tight closure, i.e. for any \fg\ $R$-modules $N \subseteq M$, $N_M^\cl \subseteq N_M^g$, where $g$ denotes generic tight closure. 
\end{prop}

\begin{proof}
To show that generic tight closure can be extended to \fg\ $R$-modules, it suffices to note that both ultraproducts and tight closure commute with finite direct sums \cite[Lemma 7.0.5]{E12}. 

Let $N \subseteq M$ be \fg\ $R$-modules. By \cite[Section 2.5]{Sch05}, they have approximations $N_w$ and $M_w$, respectively, with ultraproducts $N_\natural = \ulim\ N_w$ and $M_\natural = \ulim\ M_w$. Suppose that $u \in N_M^\cl$. Then there is some balanced big \CM\ $R$-algebra $B_\natural$ that is an ultraring, and such that 
\[1 \otimes u \in \im(B_\natural \otimes_R N \to B_\natural \otimes_R M).\]
 By Lemma \ref{ultratensor}, for any \fg\ $R$-module $Q$, we have maps 
 $$
 B_\natural \otimes_R Q \to B_\natural \otimes_{\mathfrak{D}(R)} Q_\natural \to \ulim_w\ B_w \otimes_{R_w} Q_w.
$$
 In particular, we have the following diagram:
$$
\xymatrix{
B_\natural \otimes_R N \ar[d]\ar[r] & B_\natural \otimes_R M \ar[d] \\
\ulim\ B_w \otimes_{R_w} N_w \ar[r] & \ulim\ B_w \otimes_{R_w} M_w}
$$
Since $1 \otimes u$ is in the image of the first horizontal map, $\ulim\ 1_w \otimes u_w$ is in the image of the lower horizontal map. Hence for almost every $w$, \[1_w \otimes u_w \in \im(B_w \otimes N_w \to B_w \otimes M_w).\] 
Hence, for almost all $w$, we have $u_w\in (N_w)^*_{M_w}$ over $R_w$ by \cite[Theorem 11.1]{Ho94}. This implies that $u \in N_M^g$, as desired.
\end{proof}

\begin{remk}
The reverse inclusion of above is still murky. The stumbling block to a proof is that as $B_\natural$ is not finitely generated, for a finitely generated $R$-module $Q$ we only have known maps $B_\natural \otimes_R Q \to \ulim_w\ B_w \otimes_{R_w} Q_w$, but we would need the reverse directions to prove that $N_M^g \subseteq N_M^\cl$. 
\end{remk}

As any two $B_\natural$ as above map to a common balanced big \CM\ algebra (Theorem \ref{tensorseeds}) and we have a base change property 
(Theorem \ref{basechange}), the closure operation $\cl$ above can be shown to have a host of nice properties, including: persistence, colon-capturing, phantom acyclicity,  and triviality of all closures over regular rings. Given Theorems \ref{tensorseeds} and \ref{basechange}, the proofs of these claims (and more) can be found in \cite[Chapter 7]{Dth}.

Further study of the properties of this closure operation is merited.

\section*{Acknowledgments}
Our thanks to Neil Epstein, Mel Hochster, and Hans Schoutens who each took the time to read initial drafts of this work.
We also thank the anonymous referee for many helpful suggestions that improved the exposition of the paper.

\end{document}